\documentclass{article}

\usepackage[backend=bibtex, 
sortcites=true,
firstinits=true,
hyperref,
maxbibnames=99,
]{biblatex}

\bibliography{Bibliography}{}

\usepackage{enumerate}

\usepackage[centertags]{amsmath}
\usepackage{amsfonts}
\usepackage{amssymb}
\usepackage{amsthm}
\usepackage{newlfont}
\usepackage{mathtools}
\usepackage[top=1in, bottom=1in, left=1in, right=1in]{geometry}
\usepackage{tikz}

\theoremstyle{plain}
\newtheorem{theorem}{Theorem}[section]

\newtheorem{corollary}[theorem]{Corollary}

\theoremstyle{definition}
\newtheorem{definition}[theorem]{Definition}

\newtheorem*{remark}{Remark}

\newcommand{\R}{\mathbb{R}}
\newcommand{\N}{\mathbb{N}}

\newcommand{\de}{\delta}

\newcommand{\markov}{Q}
\newcommand{\pnorm}{p}
\newcommand{\pcell}{k}
\newcommand{\kcell}{j}

\newcommand{\Deg}{\operatorname{Deg}}

\newcommand{\diam}{\operatorname{diam}}
\newcommand{\sgn}{\operatorname{sgn}}

\newcommand{\less}{\lessdot}
\newcommand{\more}{\gtrdot}

\newcommand{\eps}{\varepsilon}

\newcommand{\Hidden}[1]{}

\newcommand{\Ri}{R}

\DeclareMathOperator{\Tr}{Tr}

\newcommand{\opt}{J}

\begin{document}

\title{Characterizations of Forman curvature}
\author{
~~~~~J\"urgen Jost\footnote{MPI MiS Leipzig, jost@mis.mpg.de}~~~~~
Florentin M\"unch\footnote{MPI MiS Leipzig, muench@mis.mpg.de}
}
\date{\today}
\maketitle

\begin{abstract}
We characterize Forman curvature lower bounds via
contractivity of the Hodge Laplacian semigroup.
We prove that Ollivier and Forman curvature coincide on edges when maximizing the Forman curvature over the choice of 2-cells.
To this end, we translate between 2-cells and transport plans.
Moreover, we give improved diameter bounds.
We explicitly warn the reader that our Forman curvature notion does not coincide with Forman's original definition, but can be seen as generalization of the latter one. 
 
\end{abstract}

\tableofcontents

\newpage


\section{Introduction}
Among a not further specified group of mathematicians, the leading opinion was that the only purpose of Forman curvature is to compute at least some curvature of a huge real network when all other curvature notions are not feasible for computation due to a lack of computing power. In this article, we prove this opinion wrong.

Our first result is a characterization of lower bounds to the Forman curvature $F$ via the semigroup generated by the Hodge Laplacian $H:=\de \de^* + \de^*\de$ on a cell complex $X$ with coboundary operator $\de$. The adjoint $\delta^*$ is taken with respect to a cell weight $m$.
The \emph{Forman curvature} $F:X \to \R$ is defined as
\[
F(x):= Hx(x) - \sum_{y} |Hy(x)|
\]
with a cell $x \in X$ identified with $1_x \in C(X)$.
The definition resembles the Weitzenböck decomposition $F=H-\Delta_B$ where $\Delta_B$ is the Bochner Laplacian, and $F$ is interpreted as a multiplication operator. The decomposition is well defined by requiring that $\Delta_B$ is minimally diagonally dominant which is equivalent to a maximum principle, see Section~\ref{sec:diagDom}. 
We explicitly warn the reader that our definition of Forman curvature slightly differs from Forman's original definition in the case of weighted cell complexes. In Section~\ref{sec:OriginalForman}, we explain how to interpret our definition as a generalization of Forman's definition.

We now give the semigroup characterization of lower bounds for the Forman curvature.
\begin{theorem}
Let $(X,\de,m)$ be a cell complex, let $\pcell \in \N$ and $R \in \R$.
The following statements are equivalent:

\begin{enumerate}[(i)]
\item $F(x) \geq R$ for all $\pcell$-cells $x \in X_\pcell$,
\item $\|e^{-tH} f\|_\pnorm \leq e^{-Rt} \|f\|_\pnorm$ for all $f \in C(X_\pcell)$ and all $\pnorm \in [1,\infty]$.
\end{enumerate}
\end{theorem}
The theorem reappears as Theorem~\ref{thm:HodgeSemigroup}.
This is an important contribution as 
Forman curvature is the only discrete Ricci curvature notion, for which no semigroup characterization was known.
Our main result is that Forman curvature and Ollivier curvature coincide when maximizing the Forman curvature over the weights of the 2-cells.
The \emph{Ollivier curvature} $\kappa:X \to \R$ 
introduced in \cite{ollivier2009ricci} 
can be defined (compare \cite{munch2017ollivier}) as
\[
\kappa(x):= \inf_{\substack{\de f(x)=1\\|\de f| \leq 1}} \de \de^*\de f(x).
\]

The relation between Forman and Ollivier curvature has been studied in \cite{loisel2014ricci,tee2021enhanced}, and coincidence has been shown recently in certain special cases in \cite{tee2021enhanced}. We show coincidence in full generality. 

\begin{theorem}\label{thm:KequalsF}
Let $G=(X,\de,m)$ be a 1-dimensional cell complex. 
Then for $x \in C(X_1)$,
\begin{align*}
\kappa(x) = \max_K F_K(x)
\end{align*}
where the maximum is taken over all cell complexes $K$ having $G$ as 1-skeleton.
\end{theorem}
The theorem reappears as Corollary~\ref{cor:KequalsF}.
We give two proofs of the theorem, one via dual linear programming, one via translating transport plans to cycle weights.
We also give the dual program to maximizing over the choice of cycles the minimal Forman curvature, see Theorem~\ref{thm:maxminForman}, and the appendix for a Mathematica implementation.

The curvatures explained above are implicitly based on the combinatorial path distance. We generalize the curvature notions to make them compatible to arbitrary path distances generated by a postive function $\omega : X \to (0,\infty)$, see Section~\ref{sec:VaryDistance}. 
Particularly, we define the generalized Forman curvature
\[
F_\omega (x) :=  Hx(x) - \sum_{y \neq x} \frac{\omega(y)}{\omega(x)}|Hy(x)|.
\]
and the generalized Ollivier curvature
\[
\omega(x)\kappa_\omega(x) := \inf_{\substack{\de f(x) = \omega(x)\\|\de f| \leq \omega}} \de \de^*\de f
\]
corresponding to the path distance generated by $\omega$ on 1-cells.
We then transfer our results to the weighted case.
Finally, we recover Forman's original definition with cell weights $w$ by setting $m:=\frac 1 w$ and $\omega := \sqrt{w}$, see Section~\ref{sec:OriginalForman}.
Section~\ref{sec:VaryDistance} can be read as an extended summary of our results.

\subsection{Background}

Discrete curvature is a vibrant subject of recent research.
It plays an important role in both, pure and applied math.
In pure math, it connects differential geometry, combinatorics, global analysis and algebraic topology. 
In applied math, it is a valueable tool in network analysis \cite{sreejith2016forman,saucan2018discrete} and has been used for network alignment \cite{ni2018network}, community detection \cite{ni2019community}, and also e.g. for analyzing the topology of the internet \cite{ni2015ricci}, cancer networks \cite{sandhu2015graph}, or instabilities in financial markets \cite{samal2021network}.

The concept of discrete Ricci curvature turned out to be particularly fruitful. 
It originated from the Ricci curvature of Riemannian geometry. In fact, Ricci curvature is one of the most important concepts of Riemannian geometry (see \cite{jost2017a}). For instance, it controls the volume growth of distance balls via the Bishop-Gromov inequality. A positive lower bound of the Ricci curvature leads to an upper bound for the diameter of a complete Riemannian manifold in the Bonnet-Myers theorem. Ricci curvature  enters into the Bochner-Lichnerowicz formula from which homological restrictions or Laplacian eigenvalue bounds can be derived in the case of positive curvature. It also controls the coupling of random walks. Ricci curvature is likewise fundamental in  basic models of 20th century theoretical physics.  In the Einstein field equations of general relativity, the space-time geometry is given in terms of its Ricci curvature. Fundamenal constituents of superstring theory, the Calabi-Yau manifolds, are characterized by the vanishing of their Ricci curvature.  

Originally, Ricci curvature is defined as a contraction of the Riemann curvature tensor, and thus obtained in terms of certain combinations of second and first derivatives of the metric tensor, i.e., as an infinitesimal quantity. As indicated above, Ricci curvature bounds that hold throughout the manifold in question imply certain local or global properties of that manifold. In fact, several of these properties turned out to be equivalent to Ricci curvature bounds. Therefore, it was natural to envision more abstract definitions of Ricci curvature, or at least Ricci bounds, in terms of those properties. In particular, this led to notions of the Ricci curvature of graphs, or sometimes, more generally, simplicial or CW complexes. The most important ones seem to be 
Forman curvature \cite{forman2003bochner}, Ollivier curvarue \cite{ollivier2007ricci,ollivier2009ricci} (with important modifications or reformulations in \cite{lin2011ricci,munch2017ollivier}), Bakry Emery curvature \cite{schmuckenschlager1998curvature,lin2010ricci} and entropic Ricci curvature \cite{mielke2013geodesic,erbar2012ricci}.   While all these notions reduce to the original Ricci curvature when applied to a Riemannian manifold, for graphs they take different values and their bounds have different geometric and analytic consequences. While the latter three can be characterized by gradient estimates for the heat equation on a graph, the Forman curvature of a cell complex is based on a Weitzenb\"ock decomposition of the Hodge Laplacian as a sum of a minimally diagonally dominant operator interpreted as Bochner Laplacian, and a multiplication operator interpreted as curvature.

The question then naturally emerges how these different notions are related to each other, and in particular, how the conceptually somewhat different Forman curvature is related to the others. They cannot be identical, already for the simple reason that the Forman curvature of a graph can be computed by a very easy formula, while for instance Ollivier curvature requires the solution of a more difficult optimization problem. And the Forman curvature of a graph does not take triangles into account, while triangles strongly affect the value of the Ollivier curvature. But both Forman and Ollivier curvature are naturally assigned to the edges of a graph, and in empirical data, they show correlations \cite{samal2018comparative}. 
In this contribution, we show that (a slight reformulation of) Forman curvature becomes equal to Ollivier curvature when we also optimize it. The optimization for Ollivier requires an optimal transport plan between the neighborhoods of the vertices of an edge. The optimization for Forman that we describe here seems to be of a different nature, as we shall optimize the weights of CW complexes that have the original graph as its 1-skeleton. But in the end, the two optimizations turn out to be equivalent.



\section{Diagonally dominant non-negative operators}\label{sec:diagDom}
We start introducing weighted Euclidean spaces.
Notation for weighted Euclidean spaces is notoriously difficult. 
It is always tempting to mix up whether the weight should be in the enumerator, the denominator, or nowhere.
Let $X$ be a finite set and $m:X\to (0,\infty)$. The function $m$ induces a scalar product on $C(X)=\R^X$ by
\[
\langle f,g \rangle := \sum_{x\in X} f(x)g(x)m(x).
\]
The corresponding Hilbert space is denoted by $\ell_2(X,m)$.
By abuse of notation, we write $x=1_x \in C(X)$ for $x \in X$.
For a linear operator $A : C(X) \to C(X)$ with adjoint $A^*$, we have 
\[
Af(x) = \sum_y Ay(x) f(y) = \sum_y \frac {f(y)} {m(x)} \langle Ay,x \rangle
\]
and
\[
Ay(x) = \frac{\langle Ay,x \rangle}{m(x)}  =  \frac{\langle y,A^*x \rangle}{m(x)} = \frac{m(y)}{m(x)} A^*x(y).
\]
A selfadjoint operator on $H$ on $\ell_2(X,m)$ is called \emph{diagonally dominant} (and non-negative) if
\[
Hx(x) \geq \sum_{y\neq x} |Hy(x)|, \hspace{1cm} \mbox{ for all } x \in X,
\]
and \emph{minimally diagonally dominant} (and non-negative) if
\[
Hx(x) = \sum_{y\neq x} |Hy(x)|, \hspace{1cm} \mbox{ for all } x \in X.
\]
For convenience, we omit mentioning non-negativity. Forman calls this property strong nonnegativity \cite[Definition~1.2]{forman2003bochner}.
There is a vast amount of literature about diagonally dominant operators, see e.g. \cite{barker1975cones}.
For defining Forman curvature, one decomposes some selfadjoint $H$ as
\[
H=\Delta + D
\]
where $\Delta$ is minimally diagonally dominant, and $D$ is diagonal and will play the role of Forman curvature.
\begin{remark}
The decomposition vaguely reminds of decomposing a Schroedingeer operator into a Laplacian part $\Delta$ and a potential part $D$.
Indeed, $\Delta$ can be interpreted as the Laplacian of a signed graph. 
Moreover, diagonally non-negative operators can be interpreted as a signed graph Laplacian plus non-negative potential.
For details about signed graph Laplacians, see e.g. \cite{zaslavsky1982signed}.
\end{remark}
We now give an explicit expression for $D$.
For a selfadjoint operator $H$ on $\ell_2(X,m)$,
we define $D:=D(H):\ell_2(X,m) \to \ell_2(X,m)$ as a diagonal operator given by $Dx(y):=0$ if $x\neq y$ and
\begin{align}\label{eq:diagPlusSignedLaplace}
Dx(x) := Hx(x) - \sum_{y\neq x} Hy(x), \hspace{1cm}\mbox{ for all } x \in X.
\end{align}
We notice that $D$ is the unique diagonal operator such that $H-D$ is minimally diagonally dominant.
Applying $D$ to the later introduced Hodge Laplacian will give the Forman curvature.

\subsection{Semigroup estimate and maximum principle}
We now characterize diagonal dominance of $H$ by contractivity of the semigroup $e^{-tH}$, and by a maximum principle. To do so, we need the norms
\[
\|f\|_\pnorm^\pnorm:= \sum_{x\in X} m(x)|f^\pnorm|
\]
for $\pnorm \in [1,\infty)$ and $\|f\|_\infty := \max_{x\in X} |f(x)|$ for $f \in C(X)$. For an operator $H:C(X) \to C(X)$, we define $\|H\|_{\pnorm\to \pnorm}:= \sup_{\|f\|_\pnorm=1}\|Hf\|_\pnorm$.
\begin{theorem}\label{thm:semigroupPUnweighted}
Let $H$ be a selfadjoint operator on $\ell_2(X,m)$. 
The following statements are equivalent.
\begin{enumerate}[(i)]
\item $H$ is diagonally dominant,
\item $\|e^{-tH}\|_{\infty \to \infty} \leq 1$,
\item $\|e^{-tH}\|_{\pnorm \to \pnorm} \leq 1$ for all $\pnorm \in [1,\infty]$,
\item $Hf(x) \geq 0$ whenever $f(x)=\|f\|_\infty$.
\end{enumerate}
\end{theorem}
\begin{proof}
The implication $(iii) \Rightarrow (ii)$ is clear. We next prove $(ii)\Rightarrow (i)$.
Let $x \in X$ and $f(x):=1$ and $f(y):= -\sgn(Hy(x))$ for $y \neq x$.
Then, 
\[
-\partial_t e^{-tH}f(x)|_{t=0} = Hf(x) = Hx(x) - \sum_{y \neq x} \sgn (Hy(x)) Hy(x) = Hx(x) - \sum_{y \neq x } |Hy(x)|
\]
By $(ii)$, we have $\partial_t e^{-tH}f(x)|_{t=0} \leq 0$ implying diagonal dominance as $x \in X$ is arbitrary.
We finally prove $(i) \Rightarrow (iii)$. As the norms are continuous in $\pnorm$, we can assume $\pnorm \in (0,\infty)$.
Let $f \in C(X)$. Then with $g:= e^{-tH} f$ and $\phi(s):=|s|^{\pnorm-1}\sgn(s)$, and $\sigma_{xy} := \sgn Hy(x)$, 
\begin{align*}
-\partial_t \|e^{-tH}f\|_\pnorm^\pnorm &= \langle Hg, \phi(g) \rangle 
\\&= \langle Dg,\phi(g) \rangle + \sum_{x \neq y}\langle Hx,y \rangle g(x)\phi(g(y)) + \sum_x g(x)\phi(g(x))m(x) \sum_{y\neq x}|Hy(x)|
\\
&= \langle Dg,\phi(g) \rangle + \frac 1 2 \sum_{x \neq y} |\langle Hy,x\rangle| \Big(g(x)+\sigma_{xy}g(y) \Big) \Big(\phi(g(x)) + \sigma_{xy} \phi(g(y)) \Big) \\
&\geq 0
\end{align*}
where the estimate follows as $D \geq 0$ by diagonal dominance and as $\phi$ is odd and increasing.
Integrating over $t$ gives $\|e^{-tH}f\|_\pnorm \leq \|f\|_\pnorm$. This implies $(iii)$. 
Finally, the equivalence of $(ii)$ and $(iv)$ follows by taking derivative at $t=0$ in $(ii)$.
\end{proof}

\begin{remark}
For each diagonally dominant operator $H$, one can construct a covering graph Laplacian $H' : C(X\times \{1,-1\}) \to C(X\times \{1,-1\})$ with non-negative potential as follows.
Let $\sigma_{xy}:=-\sgn(Hy(x))$ and
\[
H' f(x,k) := Dx(x)f(x,k) + \frac {1}{m(x)}\sum_{y} |\langle Hy,x\rangle|(f(x,k) - f(y,k\sigma_{xy})). 
\]
If $H$ is minimally diagonally dominant, then $D=0$, i.e., the potential vanishes.
With \[C'(X) := \{f \in C(X\times \{1,-1\}): f(x,k)=-f(x,-k) \mbox{ for all } x\in X \mbox{ and } k \in \{-1,1\}\}\]
\end{remark}
and the isomorphism $\Phi:C'(X) \to C(X), f \mapsto f(\cdot,1)$, we get
\[
H'|_{C'(X)} = \Phi^{-1} H \Phi.
\]
Hence, semigroup contractivity and maximum principle of $H$ can be deduced by the corresponding properties of $\Delta$ which follow from standard theory on Dirichlet forms or Markov semigroups, see e.g. \cite[Lemma~3.10]{eberle2009markov}.

\subsection{Weighted diagonal dominance}

Let $H$ be selfadjoint on $\ell_2(X,m)$. Let $\omega:X \to (0,\infty)$. We say $H$ is \emph{diagonally dominant with respect to} $\omega$ (and non-negative) if
\[
\omega(x) Hx(x) \geq \sum_{y \neq x} \omega(y)Hy(x).
\]
Operators $H$ being diagonally dominant with respect to some positive $\omega$ are called $H$-matrices in the literature \cite{plemmons1977m}.
For giving a semigroup characterization, we have to modify the norms.
We define 
\[
\|f\|_{\omega,\pnorm} := \| \omega^{\frac 2 \pnorm - 1} \cdot f\|_\pnorm
\]
for $\pnorm \in [1,\infty)$ and $\|f\|_{\infty,\omega} := \|f/\omega\|_\infty$.

\begin{theorem}\label{thm:weightedDiagonalDominance}
Let $H$ be a selfadjoint operator on $\ell_2(X,m)$ and $\omega:X \to (0,\infty)$. 
The following statements are equivalent.
\begin{enumerate}[(i)]
\item $H$ is diagonally dominant with respect to $\omega$.
\item $\left\|e^{-tH}f\right\|_{\omega,\infty} \leq \left\|f \right\|_{\omega,\infty}$ for all $f \in C(X)$,
\item $\left\|e^{-tH}f\right\|_{\omega,\pnorm} \leq \left\|f \right\|_{\omega,\pnorm}$ for all $f \in C(X)$ and $\pnorm \in [1,\infty]$,
\item $Hf(x) \geq 0$ whenever $ (f/\omega)(x) = \|f\|_{\omega,\infty}$.
\end{enumerate}
\end{theorem}
\begin{proof}
We define $H':=\frac 1 \omega H \omega$ where $\omega$ is interpreted as a multiplication operator.
A short computation shows that $H'$ is selfaddjoint on $\ell_2(X,m\omega^2)$. We notice that $(i)$ is equivalent to diagonal dominance of $H'$. By Theorem~\ref{thm:semigroupPUnweighted}, this is equivalent to

\[
\langle \omega^2, |e^{-tH'} f|^\pnorm \rangle \leq  \langle \omega^2, |f|^\pnorm \rangle \mbox{ for all } f \in C(X), \pnorm \in [1,\infty)
\]
We have $e^{-tH'} = \frac 1 \omega e^{-tH} \omega$ and with $g:= \omega f$, we have
\[
\langle \omega^2, |e^{-tH'} f|^\pnorm \rangle = \langle \omega^{2-\pnorm}, |e^{-tH}g|^\pnorm\rangle  = \left\| e^{-tH}g\right\|_{\omega,\pnorm}^\pnorm
\]
and
\[
\langle \omega^2, |f|^\pnorm \rangle = \langle \omega^{2-\pnorm},|g|^\pnorm \rangle = 
\left\| g\right\|_{\omega,\pnorm}^\pnorm
\]
Putting together shows $(i) \Leftrightarrow (iii)$.
Using continuity of $\|f\|_{\omega,\pnorm}$ at $\pnorm=\infty$ and applying Theorem~\ref{thm:semigroupPUnweighted} again shows $(i) \Leftrightarrow (ii)$. 
The equivalence $(ii)\Leftrightarrow (iv)$ follows by taking derivative of $(ii)$ at $t=0$, similar to Theorem~\ref{thm:semigroupPUnweighted}.
This finishes the proof.
\end{proof}

\section{Cell complexes}
In this article, we will use a combinatorial notion of cell complex.
We aim to stay as general as possible while being compatible with weighted graphs and 2-dimensional regular CW complexes. We now give our definition of cell complex which might differ from the definitions in the literature. 
\begin{definition}
A cell complex  $K=(X,\de,m)$ consists of a finite set of cells $X=\dot \bigcup_{\pcell \geq 0} X_\pcell$, a linear operator $\de:C(X) \to C(X)$ called coboundary operator, and a positive weight function $m:X \to (0,\infty)$. For $x,z \in X$, we define
\begin{itemize}
\item $\dim(z) := k$ iff $z \in X_k$,
\item $\dim(K) := \max_{v\in X} \dim(v)$,
\item $z \more x$ iff $\de x(z) \neq 0$.
\end{itemize}
We require a cell complex to carry a hypergraph structure between $X_k$ and $X_{k+1}$. Specifically, we require that $\de$ restricted to $X_k$ is the incidence operator of an oriented hypergraph, i.e., 
for all $v,z \in X$ and all $k \in \N_0$,
\begin{itemize}
\item $\delta v(z) \in \{-1,0,1\}$, 
\item $\delta: C(X_k) \to C(X_{k+1})$, i.e., $\delta v(z)=0$ if $\dim(z) - \dim(v) \neq 1$,
\item If $\dim(z) \geq 1$, then there exists $x \less z$.
\end{itemize}
Moreover, we require a cell complex to satisfy the following compatibility conditions. For all $v,z,z' \in X$,
\begin{enumerate}[(i)]
\item 
$|\{w \in X_0 :\de w(x) = 1\}| = |\{w \in X_0 :\de w (x) = -1\}|=1$ for all $x \in X_1$,
\item If $\dim(z) - \dim(v)=2$ and $\{x: v\less x \less z\} \neq \emptyset$, then,
  \begin{enumerate}
   \item  $|\{x: v \less x \less z\}|=2$,
   \item $\{\de v(x) \de x(z):  x \in X\} \supseteq \{-1,1\}$,
  \end{enumerate}
\item For all $x,y \less z$, there is a sequence $(x=x_0,\ldots,x_n=y)$ with $\de v_k(v_{k-1}) + \de v_{k-1}(v_{k}) \neq 0$ for all $k=1,\ldots, n$,
\item   If $\{x:x\less z\} = \{x:x \less z'\}$ and $\dim(z) \geq 1$, then $z=z'$.
\end{enumerate}

\end{definition}

Condition (i) guarantees that $X_1$ behaves edge like, i.e., every $x \in X_1$ is incident to exactly two vertices $v \in X_0$. Without (i), every hypergraph could be seen as a cell complex.
Condition (ii)(a) is also known as diamond property, and (ii)(b)  guarantees that $\de^2 =0$.
Condition (iii) is a connectedness condition ensuring that every 2-dimensional cell complex is a regular CW complex.
Condition (iv) ensures that every 1-dimensional cell complex is a simple graph, and more generally, that cells do not appear twice.

We call $x\in X_\pcell$ a $\pcell$-cell. We call $Y:=\bigcup_{\kcell\leq \pcell} X_\kcell$, together with the corresponding restrictions of the coboundary $\de$ and the weight $m$, the $\pcell$-skeleton of $(X,\de,m)$.

For $x,y \in X$, we write $x\sim y$ if there exists $z\more x,y$ or $v\less x,y$. This induces a path distance on $X$ via
\[
d(x,y) = \inf\{n:x=x_0 \sim \ldots \sim x_n=y\}.
\]
We notice that $d(x,y)$ can be infinite. We will only need the distance on the vertex set $X_0$.

We observe
\[
\de \de v(z) = \sum_{v\less x \less z} \de v(x)\de x(z) = 0
\]
by condition (ii).
We equip $X$ with the scalar product
\[
\langle f,g \rangle := \sum_{x \in X} f(x)g(x)m(x) \quad \mbox{for }f,g \in C(X).
\]
By this, we get the adjoint $\de^*:C(X) \to C(X)$,
\[
\de^*z(x) = \frac{1}{m(x)} \langle \de^*z,x \rangle = \frac{1}{m(x)}\langle \de x,z \rangle = \frac{m(z)}{m(x)}\de x(z) \quad \mbox{for }x,z \in X
\]
and by linear extension,
\[
\de^*f(x) = \sum_z f(z)\de^*z(x) = \sum_{z\more x} f(z)\frac{m(z)}{m(x)}\de x(z).
\]
Restricting to $C(X_{\pcell+1})$ gives $\de^*:C(X_{\pcell+1}) \to C(X_\pcell)$.

We remark that our definition slightly differs from Forman's definition as he defines $\partial: C(X_k) \to C(X_{k-1})$. Particularly with Forman's notation, one has
$\partial z(x) = \de x(z) =   \frac{m(z)}{m(x)}\de^*z(x)$.
Our definition has the advantage, that it is compatible with the standard notation for weighted graphs, as shown later.

\subsection{Hodge Laplacian}

We define the Hodge Laplacian $H:C(X) \to C(X)$ given by
\[
H=\de \de^* + \de^* \de.
\]
We write $H_\pcell := H|_{C(X_\pcell)} : C(X_\pcell) \to C(X_\pcell)$.

We now calculate $Hy(x)$.
We have
\[
\de \de^*y(x) = \frac{1}{m(x)}\langle \de^*y,\de^*x \rangle = \frac{1}{m(x)} \sum_{v\less x,y} m(v)   \frac{m(y)}{m(v)}\de v(y)  \frac{m(x)}{m(v)}\de v(x)  = \sum_{v \less x,y}\frac{m(y)}{m(v)}\de v(x)\de v(y)
\]
and
\[
\de^*\de y(x) = \frac{1}{m(x)}\langle \de y,\de x \rangle = \frac{1}{m(x)} \sum_{z\more x,y} m(z)   \de y(z)  \de x (z) = \sum_{z \more x,y}\frac{m(z)}{m(x)}\de x(z) \de y(z).
\]
Hence
\[
Hy(x) =  \sum_{v \less x,y}\frac{m(y)}{m(v)}\de v (x)\de v(y)  +   \sum_{z \more x,y}\frac{m(z)}{m(x)}\de x(z) \de y(z).
\]
We remark that, in contrast to Forman's definition of the Hodge Laplacian, the weight of the smaller cell is always in the denominator.

If $x,y \in X_1$, then the first sum consists of at most one element, and by condition (ii), we have $\de v(x)\de v(y) = -\de x(z) \de y(z)$ for all $v \less x,y$ and $z \more x,y$. Hence for $x,y \in X_1$,
\begin{align}\label{eq:absH}
|Hy(x)| = \left| \sum_{v \less x,y}\frac{m(y)}{m(v)}  -   \sum_{z \more x,y}\frac{m(z)}{m(x)} \right|.
\end{align}

Moreover,
\begin{align}\label{eq:Hdiag}
Hx(x) = \sum_{v\less x} \frac{m(x)}{m(v)} + \sum_{z \more x} \frac{m(z)}{m(x)}.
\end{align}

\subsection{Weighted graphs}
We now show how to interpret 1-dimensional cell complexes as weighted graphs.
We particularly demonstrate how $H$ on $C(X_0)$ is the usual weighted graph Laplacian.

For every $x \in X_1$, there is a unique pair $(v,w)$ with $\de v(x)=-1$ and $\de w(x)=1$. By abuse of notation, we write $x=(v,w)=(w,v)$.
We then also write $m(w,v) := m(v,w) := m(x)$, and we write $m(v,w):= 0$ if there is no $x \more v,w$. We also write $m(v,v)=0$ for $v \in X_0$. 
We notice that  for $v,w \in X_0$ and $x \more v,w$,
\[
\de v(x)\de w(x) = \begin{cases} 1&:v=w, \\
-1&:v\neq w.
\end{cases}
\]

For $f \in C(X_0)$, we have $\de^*f = 0$ and thus for $v \in X_0$,
\begin{align*}
Hf(v) = \de^* \de f(v) &= \sum_w \sum_{x \more v,w} \frac{m(x)}{m(v)}\de v (x) \de w(x) f(w) \\
&= \sum_{x \more v} \frac{m(x)}{m(v)} \sum_{w \less x} \de v (x)\de w(x)f(w) \\
&=\sum_{w \in X_0} \frac{m(v,w)}{m(v)}(f(v)-f(w)).
\end{align*}

This is clearly the positive weighted graph Laplacian.
Particularly, the above procedure gives a one-to-one correspondence between weighted graphs and 1-dimensional cell complexes.
Therefore, we will usually write $G$ instead of $K$ for 1-dimensional cell complexes.

\subsection{Extending the 1-skeleton}

Let $(X,\de,m)$ be a 1-dimensional cell complex.
We recall, for $v,w \in X_0$, we have $v \sim w$ iff there exists $x \more v,w$.
A cycle is an injective path $(v_0 \sim \ldots \sim v_{n-1})$ of vertices $v_i \in X_0$ with $v_0 \sim v_{n-1}$ with $n \geq 3$.

Two cycles $(v_i)_{i=0}^{n-1}$ and $(w_i)_{i=0}^{n-1}$ are identified if $w_i = v_{k \pm i \mod n}$ for some $k \in \N$ and some fixed choice of plus or minus. 

Let $Y_2$ be the set of cycles and let $m:Y_2 \to [0,\infty)$. We write $X_2 := \{x \in Y_2: m(x)>0\}$.

For $x=(v,w) \in X_1$ and $z=(v_0 \sim \ldots \sim v_{n-1})\in X_2$, we set
\[
\de x(z) := \begin{cases}1&: v=v_k,w=v_{k+1\mod n} \mbox{ for some } k \in \N\\
-1&:v=v_k, w=v_{k-1 \mod n} \mbox{ for some } k \in \N\\
0&:\mbox{ else}.
\end{cases}
\]

Then, $(X \cup X_2,\de,m)$ is a cell complex, and also a regular weighted CW complex.
Conversely if $X$ is a cell complex of dimension at least two, then its 2-skeleton can be described as above.


\section{Forman curvature}

Forman's idea was to decompose the Hodge Laplace $H$ as a sum of a minimally diagonally dominant operator serving as Bochner Laplacian, and a diagonal operator, now known as Forman curvature \cite{forman2003bochner}.
We proceed executing Forman's ideas in a somewhat different way.
According to \eqref{eq:diagPlusSignedLaplace}, the diagonal operator  is given by
\begin{align*}
F(x):=D(H)x(x) &= Hx(x) - \sum_{y\neq x} |Hy(x)| \\
&= \sum_{v \less x}\frac{m(x)}{m(v)} + \sum_{z \more x} \frac{m(z)}{m(x)} - \sum_{y\neq x} \left|   \sum_{v \less x,y} \de v(x)\de v(y) \frac{m(y)}{m(v)} + \sum_{z\more x,y} \de x(z) \de y(z) \frac{m(z)}{m(x)}  \right|.
\end{align*}

If $x \in X_1$, the expression can be simplified via \eqref{eq:absH} and \eqref{eq:Hdiag} to
\[
F(x)= \sum_{v \less x}\frac{m(x)}{m(v)} + \sum_{z \more x} \frac{m(z)}{m(x)} - \sum_{y\neq x} \left|   \sum_{v\less x,y} \frac{m(y)}{m(v)} - \sum_{z \more x,y}  \frac{m(z)}{m(x)}  \right|.
\]

This is our new notion of a Forman type curvature with weights consistent with weighted graph Laplacians. For convenience, we simply call it Forman curvature although it slightly differs from Forman's original definition, see Section~\ref{sec:OriginalForman}.

We remark that Forman used the same approach of splitting off a minimally diagonally dominant operator to define curvature, however he used an alleged symmetrization of $H$ before the splitting off. This symmetrization made $H$ a symmetric matrix, i.e., self-adjoint with respect to the constant weight.
From this, he split off a  minimally diagonally dominant operator with respect to the constant weight. In contrast, our splitting off works with the given weight $m$. Consequently, our modified Forman curvature coincides with the original one in the unweighted case $m\equiv 1$, but not in the general case. In Section~\ref{sec:OriginalForman}, it is discussed how Forman's definition can be interpreted as a special case of our definition.

\subsection{Characterizing Forman curvature via the Hodge semigroup}

We now characterize lower Forman curvature bounds via contractivity of the Hodge semigroup $e^{-Ht}$. 

\begin{theorem}\label{thm:HodgeSemigroup}
Let $(X,\de,m)$ be a cell complex, let $\pcell \in \N$ and $\Ri \in \R$.
The following statements are equivalent:

\begin{enumerate}[(i)]
\item $F(x) \geq \Ri$ for all $x \in X_\pcell$,
\item $\|e^{-tH} f\|_\infty \leq e^{-\Ri t} \|f\|_\infty$ for all $f \in C(X_\pcell)$,
\item $\|e^{-tH} f\|_\pnorm \leq e^{-\Ri t} \|f\|_\pnorm$ for all $f \in C(X_\pcell)$ and all $p \in [1,\infty]$,
\item $H_\pcell-R$ is diagonally dominant and non-negative.
\end{enumerate}
\end{theorem}

\begin{proof}
The theorem follows immediately from Theorem~\ref{thm:semigroupPUnweighted} as $F=D(H)$, meaning that $H-F$ is minimally diagonally dominant.
\end{proof}

\subsection{Maximizing Forman curvature over higher order cell weights}

In this section, we demonstrate how maximizing the Forman curvature over higher order cell weights can be interpreted as a linear program. 
We will give the dual linear program which we will use as interface to the Ollivier curvature.
For specifying over which cell complexes to maximize the Forman curvature, we introduce the relation $\leq_\pcell$ on cell complexes.

Let $K=(X,\de,m)$ and $K'=(X',\de',m')$ be cell complexes.

We write $K' \leq_\pcell K$ if
\begin{enumerate}[(i)]
\item
$X_\kcell = X'_\kcell$ for $\kcell\leq \pcell$,
\item 
$X_\kcell' \subseteq X_\kcell$ for $\kcell >\pcell$,
\item $m(x)=m'(x)$ whenever $\dim(x) \leq \pcell$, and
\item $\de'v(x)=\de v(x)$ for all $v,x \in X'$.
\end{enumerate}

Roughly speaking, $K' \leq_\pcell K$ means that their $\pcell$-skeletons coincide and that all higher order cells of $K'$ also belong to $K$, however the higher order cell weights can differ.

We now give the dual expression for the Forman curvature maximized over all $K'\leq_\pcell K$ for a fixed cell complex $K$.
\begin{theorem}\label{thm:dualForman}
Let $K=(X,\de,m)$ be a cell complex and $x \in X_\pcell$ for some $\pcell \in \N_0$.
Then,
\[
\max_{K' \leq_\pcell K}F'(x) = \min_{\substack{h(x)=1\\|h|\leq 1\\\de h \cdot \de x \leq 0}} \de \de^*h(x).
\]

\end{theorem}

\begin{proof}
W.l.o.g., $m(y)=1$ whenever $\dim(y)>\dim(x)$.
Hence for $z \in X_{\pcell+1}$ and $y \in X_\pcell$,
\[
(\de')^*z(y) = \frac{1}{m(y)} \langle (\de')^*z,y \rangle = \frac{1}{m(y)} \langle z, \de y \rangle' = \frac{m(z)}{m(y)} \langle z, \de y \rangle = m'(z)\de^*z(y) = \de^* m'z(y).
\]
We write $n:=m'|_{X_{\pcell + 1}}$.
Then on $X_\pcell$,
\[
H_\pcell' = \de \de^* + \de^*n\de.
\]
Hence,
\begin{align*}
\max_{K'\leq_\pcell K} F'(x) &= \max_{K' \leq_\pcell K}  H'x(x) - \sum_{y\neq x} |H'y(x)| \\
&= \max_{\substack{n \in C(X_{\pcell + 1})\\n\geq 0}} \left( \de \de^*x(x) + \de^*n\de x(x) - \sum_{y\neq x} \left| \de \de^*y(x) + \de^*n\de y(x) \right| \right).
\end{align*}

We write this as a linear program by replacing the absolute values by a variable $g(y)$ for $y \in X_\pcell \setminus \{x\}$.
The variables are $n(z) \geq 0$ for $z \in X_{\pcell + 1}$ and $g(y) \in \R$.
The linear program is

\vspace{0.5cm}
\noindent
Maximize
\[
\de \de^*x(x) + \de^*n\de x(x) - \sum_{y\neq x} g(y)
\]
subject to
\begin{align*}
g(y) - \de^*n\de y(x) &\geq \de \de^*y(x),  &y \in X_\pcell \setminus \{x\}& \qquad (c_-(y)\leq 0)\\
g(y) + \de^*n\de y(x) &\geq -\de \de ^*y(x), &y \in X_\pcell \setminus \{x\}& \qquad (c_+(y)\leq 0)
\end{align*}

For convenience, we already wrote the dual variables in parentheses.
The dual linear program consists of variables $c_-(y),c_+(y) \leq 0$ for $y \in X_\pcell \setminus \{x\}$. It is given by 

\vspace{0.5cm}
\noindent
Minimize
\[
\de \de ^*x(x) +  \de \de^*c_-(x) - \de \de^*c_+(x)
\]
subject to
\begin{align*}
c_-(y)+c_+(y) &= -1,&y \in X_\pcell \setminus \{x\}  &\qquad (g(y) \in \R)\\
\de x(z) \de(c_+-c_-)(z) &\geq (\de x(z))^2,& z \in X_{\pcell + 1}  & \qquad (n(z) \geq 0)
\end{align*}

Again, we wrote the primal variables corresponding to the dual conditions in parentheses.
We define $h \in C(X_\pcell)$ via $h=c_--c_+ + 1_x$.
Then, $|h(y)| \leq 1$ is an equivalent condition to $c_-(y)+c_+(y)=-1$ and $c_-,c_+ \leq 0$.

Moreover, the term to minimize becomes $\de \de^*h(x)$, and the second condition becomes $\de x(z)\de h(z) \leq 0$.
As the primal and the dual linear program have the same optimal value, we get
\[
\max_{K' \leq_\pcell K} F'(x) = \min_{\substack{h(x)=1 \\ |h|<1 \\ \de x  \cdot \de h\leq 0}}\de \de^*h(x).
\]
This finishes the proof.
\end{proof}

We now characterize the optimal lower bound for the Forman curvature when maximizing over the cell weights.
For an operator $\opt: C(X_\pcell) \to C(X_\pcell)$, we write $\Tr \opt := \sum_{x \in X}\opt x(x)$.
A Mathematica implementation of the following theorem can be found in the appendix.
\begin{theorem}\label{thm:maxminForman}
Let $K=(X,\delta,m)$ be a cell complex and $k \in \N_0$. Then,
\[
\max_{K' \leq_k K} \min_{x \in X_k}F'(x) = \min_\opt \Tr(\de \de^* \opt)
\]
where the minimum is taken over all linear $\opt \in C(X_k) \to C(X_k)$ satisfying
\begin{enumerate}[(a)]
\item $\frac{\opt x(x)}{m(x)} \geq \frac{|\opt y(x)|}{m(y)}$ for all $x,y \in X_k$,
\item $\de \opt \de^* z(z) \leq 0$ for all $ z \in X_{k+1}$,
\item $\Tr(\opt)=1$.
\end{enumerate}
Moreover, condition ${(a)}$ is equivalent to
\begin{enumerate}[(a')]
\item  $\opt f(x) \geq 0$ whenever $2\langle x,f\rangle \geq \|f\|_1$.
\end{enumerate}

\end{theorem}

\begin{proof}
As in the proof of Theorem~\ref{thm:dualForman}, we can write $H_k' = \de \de^* + \de^*n\de$ for some non-negative $n \in C(X_{k+1})$. Hence,
\[
m(x)H_k'x(x) = \langle n, (\de x)^2 \rangle + \langle\de^*x,\de^*x \rangle.
\]
We express the maximization of the Forman curvature as a linear program. To this end, we encode $m(x)|Hy(x)|$ as $g(x,y)$. The linear program contains the variables $R \in \R$ and $0 \leq n \in C(X_{k+1})$ and $g(x,y) \in \R$ for $x \neq y \in X_k$ and is given by the following.

\vspace{0.5cm}
\noindent
Maximize
\[
R
\]
subject to
\begin{align*}
&\langle n, (\de x)^2 \rangle - \sum_y &&g(x,y)  - Rm(x) &&\geq - \langle\de^*x,\de^*x \rangle, &&x \in X_k &&\qquad (R(x) \leq 0)  
\\
&-\langle n, \de x \de y \rangle + &&g(x,y) &&\geq  \langle\de^*x,\de^*y \rangle  && x\neq y \in X_k && \qquad (c_-(x,y) \leq 0)
\\
&\langle n, \de x \de y \rangle + &&g(x,y) &&\geq  -\langle\de^*x,\de^*y \rangle  && x\neq y \in X_k && \qquad (c_+(x,y) \leq 0)
\end{align*}
The first condition encodes the curvature bound $F'(x) \geq R$ for each $x$, and the latter two conditions encode $g(x,y)=m(x)|Hy(x)|$.
The dual variable names are written in parentheses behind every condition.

The dual linear program is

\vspace{0.5cm}
\noindent
Minimize
\[
\sum_{x \neq y}  (c_-(x,y) - c_+(x,y))  \langle \de^*y,\de^*x \rangle  - \sum_x   R(x)\langle \de^*x,\de^*x \rangle
\]
subject to
\begin{align*}
- R(x)  + c_-(x,y) + c_+(x,y)   &=    0          &&(g(x,y) \in \R) \\
\sum_x    (\de x)^2(z)R(x) +  \sum_{x\neq y}   \de y(z) \de x(z) \left( c_+(x,y) - c_-(x,y)  \right)   &\geq    0		&&(n(z) \geq 0) \\
-\sum_x R(x) m(x) &= 1   &&(R \in \R)
\end{align*}
We set
\[
\opt y(x) := m(y)\left(c_-(x,y) - c_+(x,y) - R(x)1_{x=y} \right).
\]
Then, the term to minimize becomes
\[
\sum_{x,y}\frac{\opt y(x)}{m(y)}\langle \de^*x, \de^*y \rangle = \sum_y \de \de^*\opt y(y) = \Tr(\de \de^*\opt).
\]
The first condition is equivalent to
\[
\frac{\opt x(x)}{m(x)} \geq \frac{|\opt y(x)|}{m(y)} \mbox{ for all } x,y \in X_k.
\]
The second condition becomes
\[
0 \geq \sum_{x,y}\frac{\opt  y(x)}{m(y)}\de x(z) \de y(z) = \sum_y \de \opt y(z) \frac{\delta y(z)}{m(y)} = \frac 1 {m(z)}\sum_y \de \opt y(z) \de^*z(y) = \frac 1 {m(z)} \de \opt \de^* z (z)
\]
for all $z \in X_{k+1}$. The third condition becomes
\[
1=\sum_x \opt x(x) = \Tr (\opt).
\]
Putting together finishes the proof of the characterization of the optimal Forman curvature lower bound. We finally show that $(a)$ and $(a')$ are equivalent. The implication $(a') \Rightarrow (a)$ follows by setting $f = x/m(x) \pm y/m(y)$. We now prove the implication $(a) \Rightarrow (a')$.
We have
\begin{align*}
\opt f(x) \geq \frac{\opt x(x)}{m(x)} m(x)f(x) - \sum_{y\neq x} \frac{|\opt y(x)|}{m(y)} m(y)|f(y)| &\geq \frac{\opt x(x)}{m(x)} \left( m(x)f(x) - \sum_{y\neq x}m(y)|f(y)| \right) \\
&\geq \frac{\opt x(x)}{m(x)} \left( 2m(x)f(x) - \sum_{y}m(y)|f(y)| \right) \\
&=\frac{\opt x(x)}{m(x)} \left(2\langle x,f \rangle - \|f\|_1 \right)\\
&\geq 0
\end{align*}
whenever $2\langle x,f \rangle \geq \|f\|_1$.
This proves $(a) \Rightarrow (a')$ and finishes the proof.
\end{proof}

\section{Ollivier curvature}

Ollivier curvature was introduced by Yann Ollivier as a Ricci curvature notion for random walks with a given distance \cite{ollivier2007ricci,ollivier2009ricci}. 
Ollivier's original definition considers discrete time random walks. Lin, Lu, and Yau adapted Ollivier's definition by taking a limit of lazy random walks \cite{lin2011ricci}. The relation between the lazyness and the curvature has been investigated in \cite{bourne2018ollivier}. The relation between curvature and the geometrically descriptive clustering coefficient is demonstrated in \cite{jost2014ollivier}.
Later, Lin, Lu and Yau's modification of Ollivier curvature was generalized to weighted graphs in \cite{munch2017ollivier}.
We now explain how to adapt the definition to cell complexes.
Let $(X,\de,m)$ be a cell complex. Its 1-skeleton is a weighted graph.
According  to \cite[Theorem~2.1]{munch2017ollivier} and using $H_0 = \de^* \de$, the Ollivier curvature of an edge $x \in X_1$ can be written as
\[
\kappa(x) = \inf_{\substack {\de f(x)=1 \\|\de f| \leq 1 
\\ f \in C(X_0)
}} \de \de^* \de f(x). 
\]
We note that we can drop the condition $f \in C(X_0)$ as $\de \de^* \de f(x)=0$ for all $f \in C(X_\pcell)$ with $\pcell>0$.
The generalization to all cells $x \in X$ is immediate.
For $x \in X$, we define
\[
\kappa(x) := \inf_{\substack {\de f(x)=1 \\|\de f| \leq 1 
}} \de \de^* \de f(x). 
\]

\subsection{A one-form characterization of Ollivier curvature}

In the definition of Ollivier curvature of an edge $\kappa(x)=\inf_{\de f(x)=1, |\de f|=1} \de \de^*\de f(x)$, the function $f \in C(X_0)$ only appears as $\de f$. So we obviously get
\[
\kappa(x)= \inf_{\substack{h(x)=1 \\ |h| \leq 1 \\ h \in \de(C(X_0))}}\de \de^* h(x)
\]
In the following theorem, we give a similar expression for Ollivier curvature, however we replace the condition that $h$ is in the range of $\de$ by the condition $\de x \de h \leq 0$.

\begin{theorem}\label{thm:OllivierVectorField} Let $(X,\de,m)$ be a cell complex. Let $x \in X_1$.
Suppose all cycles containing $x$ and having length at most five, are 2-cells. 
Then,
\[
\kappa(x) = 
\inf_{\substack{h(x)=1\\|h|\leq 1 \\ \de x \cdot \de h\leq 0}} \de \de^*h(x).
\]
For the infimum, we only need to consider $h \in C(X_1)$. 
\end{theorem}

\begin{proof}
We first prove $"\geq"$.
Let $f \in C(X_0)$ with $df(x)=1$ and $|\de f| \leq 1$ and $\kappa(x) = \de\de^*\de f(x)$.
Let $h:=\de f$. Then, $\de h=0$ and thus,
\[
\kappa(x) = \de\de^*g(x) \geq  \inf_{\substack{h(x)=1\\|h|\leq 1 \\ \de x \de h \leq 0}} \de \de^*h(x),
\]
proving $"\geq"$.

We finally prove $"\leq"$.
We note that the infimum is attained due to compactness. Let $h \in C(X_1)$ such that the infimum is attained.

We now aim to construct $f \in C(X_0)$ with $|\de f|\leq 1$ and $\de  f(x)=1$ and $\de \de^*\de f(x) \leq \de \de^*h(x)$ which would finish the proof.
We write $x=(v,w)$.
We define $h(v',v'):=0$ for all $v' \in X_0$.

We now define for $w' \in X_0$,
\[
f(w') := \min_{v' \in B_1(v)} h(v,v') + d(w',v')
\]
Clearly, $|\de f|\leq 1$ as $f$ is defined as a minimum of Lipschitz functions.

By $|h|\leq 1$, we get $f(v)=0$. For $v' \sim v$, we have $f(v') - f(v) = f(v') \leq h(v,v')$.
We now claim that for all $w' \in B_1(w)$,
\[
f(w') \geq 1 + h(w,w')
\]

Let $v' \in B_1(v)$ with $f(w') = h(v,v') + d(v',w')$. If $w'=w$ and $v' = v$, then we are done. If $d(v',w') \geq 3$, then, $f(w') \geq 2$ and we are also done.
Otherwise, there exists a shortest cycle $z$ containing the path  $(v',v,w,w')$.
As $\de x(z)\de h(z) \leq 0$ and $|\de h| \leq 1$, we obtain $h(v',v)+h(v,w) + h(w,w')\leq d(v',w')$ which is equivalent to
\[
f(w')=h(v,v') + d(v',w') \geq 1 + h(w,w')
\]
proving the claim $f(w') \geq 1 + h(w,w')$.
As $f(v)=0$, we get $f(w) \leq g(v,w)+d(w,w)=1$, and together with $f(w') \geq 1 + h(w,w')$, we get $f(w)=1$.
Hence, $f(w')-f(w) \geq h(w,w')$ for $w'\sim w$.

Putting everything together, we get $\partial f(x)=f(w)-f(v)=1$ and with $\markov(v,v'):=m(v,v')/m(v)$,
\begin{align*}
\de \de^*\de f(x) &= \de^*\de f(w) - \de^*\de f(v)  \\
&=\sum_{w'} \markov(w,w')(f(w)-f(w')) - \sum_{v'}\markov(v,v')(f(v)-f(v'))\\
& \leq \sum_{w'} \markov(w,w')g(w,w') - \sum_{v'}\markov(v,v')g(v,v')\\
&=\de \de^* g(x).
\end{align*}
This shows $\de\de^*\de f(x) \leq \de\de^*g(x)$ for some $f \in C(X_0)$ with $\de f(x)=1$ and $|\de f|\leq 1$. This proves "$\leq$" and finishes the proof of the theorem.
\end{proof}

\subsection{A new transport plan for Ollivier curvature}

This section will provide the basis of translating
transport plans to cycle weights. As we have already shown the dual linear program characterization of Forman curvature, this section is not necessary for understanding the coincidence of Forman and Ollivier curvature.

By the Kantorovic duality, the Ollivier curvature can be expressed via optimal transport plans.
Let $(X,\de,m)$ be a 1-dimensional cell complex.
We recall $\markov(v,w)=m(v,w)/m(v)$ for $v,w \in X_0$. 
Let $x=(v,w) \in X_1$.
Then by  \cite[Proposition~2.4]{munch2017ollivier},
\[
\kappa(x) = \sup_{\xi} \sum\xi(v',w')(1-d(v',w'))
\]
where the supremum is taken over all transport plans $\xi:X_0^2 \to [0,\infty)$ satisfying
\[
\sum_{w'} \xi(v',w') = \markov(v,v') \mbox{ for all } v' \neq v
\]
and
\[
\sum_{v'} \xi(v',w') = \markov(w,w')  \mbox{ for all } w' \neq w.
\]

In this section we give a new transport plan characterization of Ollivier curvature. The idea is to replace the restrictions of the old transport plan by penalty terms for the new unrestricted transport plan. We first give some definitions

We recall $v\sim w$ for $v,w \in X_0$ if there exists $x \more v,w$.
We write  $S_1(v) = \{w \in X_0:w \sim v\}$ for the sphere and $B_1(v)=\{v\} \cup S_1(v)$ for the ball with respect to the path distance $d$.

Let $v \sim w \in X_0$. 
We write $N(v,w) := S_1(v) \setminus \{w\}$ and

\[\Pi :=\Pi(v,w) := [0,\infty)^{N(v,w)\times N(w,v)}.\]
For $v' \in N(v,w)$ and $\rho \in \Pi$, we write
\[
A(v') :=A_\rho(v') := \markov(v,v') - \sum_{w'} \rho(v',w') 
\]
and for $w'  \in N(w,v)$, we write
\[
B(w') :=B_\rho(w')  := \markov(w,w') - \sum_{v'} \rho(v',w'). 
\]

The terms $A$ and $B$ can be seen as penalty terms for $\rho$ for deviating from a transport plan in the sense above. We now give the new transport plan characterization of Ollivier curvature.

\begin{theorem}\label{thm:NewTransportplan}
Let $(X,\de,m)$ be a cell complex.
Let $x=(v,w) \in X_1$. Then,
\[
\kappa(x) = \markov(v,w)+\markov(w,v) + \sup_{\rho \in \Pi} \left( \sum_{v',w'}\rho(v',w')(1-d(v',w')) - \sum_{v'} |A_\rho(v')| - \sum_{w'} |B_\rho(w')| \right).
\]
\end{theorem}

\begin{proof}
We first prove "$\kappa \geq \;$".
Let $\rho \in \Pi$. Let $f \in Lip(1)$ with $f(w)-f(v)=1$ and
$\kappa(v,w) = \Delta f(v) - \Delta f(w)$.

Then,
\begin{align*}
\Delta f(v) &= \markov(v,w) + \sum_{v' \in N(v,w)} \markov(v,v') (f(v')-f(v)) \\
&= \markov(v,w) + \sum_{v'} A(v')(f(v')-f(v)) + \sum_{v',w'}  \rho(v',w')  (f(v') - f(v)).
\end{align*}
Similarly,
\[
\Delta f(w) = \markov(w,v) + \sum_{w'} A(w')(f(w') - f(w)) + \sum_{v',w'} \rho(v',w')(f(w')-f(w)).
\]
Hence,
\begin{align*}
\Delta f(v)-\Delta f(w) &= \markov(v,w) + \markov(w,v) +  \sum_{v'} A(v')(f(v')-f(v))  - \sum_{w'}B(w')(f(w')-f(w))   \\& \hspace{2.9cm} +\sum_{v',w'}\rho(v',w')(f(v')-f(w') + 1)
\\&\geq \markov(v,w) + \markov(w,v) - \sum_{v'}|A(v')| - \sum_{w'}|B(w')|  + \sum_{v',w'} \rho(v',w')(1-d(v',w'))
\end{align*}
where we used $f \in Lip(1)$ in the estimate.
Taking supremum over $\rho \in \Pi$ proves "$\kappa \geq \;$".

We finally prove "$\kappa \leq \;$".

Let $\xi :X_0^2 \to [0,\infty)$ be an optimal transport plan, i.e.
$\sum_{w'}\xi(v',w') =\markov(v,v')$ for all $v' \in S_1(v)$ and
$\sum_{v'} \xi(v',w') = \markov(w,w')$ for all $w' \in S_1(w)$ and
\[
\kappa(v,w) = \sum_{v',w'} \xi(v',w')(1-d(v',w')).
\]
For $v' \in N(v,w)$ and $w' \in N(w,v)$, we define
\[
\rho(v',w') := \xi(v',w') + (\xi(v',w) + \xi(v,w'))\cdot 1_{v'=w'}.
\]
We calculate
\[
A(v') = \markov(v,v') - \sum_{w' \in N(w,v) }\xi(v',w')  - \sum_{w' \in N(w,v)} (\xi(v',w) + \xi(v,w')) \cdot 1_{v'=w'} = \xi(v',w)   -  (\xi(v',w) + \xi(v,v'))1_{v'\sim w}
\]
giving
\[
|A(v')| = \begin{cases}
\xi(v,v') &:v' \sim w \\
\xi(v',w) &: \mbox{else}.
\end{cases}
\]
Similarly for $w \in N(w,v)$,
\[
|B(w')| = \begin{cases}
\xi(w',w) &:w' \sim v \\
\xi(v,w') &: \mbox{else}.
\end{cases}
\]
Adding up gives
\[
\sum_{v'} |A(v')| + \sum_{w'} |B(w')| =  \sum_{v' \in N(v,w)} \xi(v',w) + \sum_{w' \in N(w,v)} \xi(v,w')
\]
Moreover,
\[
\sum_{v',w'}\rho(v',w')(1-d(v',w')) = \sum_{\substack{v'\in N(v,w) \\ w' \in N(w,v)}}  \xi(v',w')(1-d(v',w'))   +  \sum_{v' \in N(v,w) \cap N(w,v)} \left(\xi(v',w) + \xi(v,v') \right).
\]
Combining the latter sum with the sum of $|A|$ and $|B|$ gives
\begin{align*}
 &\sum_{v' \in N(v,w) \cap N(w,v)} \left(\xi(v',w) + \xi(v,v') \right)-\sum_{v'} |A(v')| - \sum_{w'} |B(w')|  \\=&  \sum_{v' \in N(v,w)} \xi(v',w) (1-d(v',w)) +  \sum_{w' \in N(w,v)} \xi(v,w')(1-d(v,w'))
\end{align*}
where we used $d(v',w) = 1$ if $w'$ is a common neighbor of $v$ and $w$ and $d(v',w) = 2$ otherwise for $v' \sim v$, and a similar relation for $d(v,w')$.
We summarize
\begin{align*}
&\sum_{v',w'}\rho(v',w')(1-d(v',w'))  -\sum_{v'} |A(v')| - \sum_{w'} |B(w')| \\=&\left( \sum_{\substack{v' \in N(v,w)\\w' \in N(w,v)}}  +  \sum_{\substack{v' \in N(v,w)\\w'=w}} + \sum_{\substack{w' \in N(w,v)\\v'=v}} \right) \xi(v',w')(1-d(v',w')) \\
=& \sum_{v',w'}\xi(v',w')(1-d(v',w')) - \xi(v,v)-\xi(w,w)  \\
\geq& \kappa(v,w) - \markov(w,v) - \markov(v,w)
\end{align*}
where we used $\xi(v,v) \leq \markov(w,v)$ and $\xi(w,w) \leq \markov(v,w)$ in the  estimate.
Rearranging proves "$\kappa \leq \;$" and finishes the proof.
\end{proof}

\section{Relations between Ollivier and Forman curvature}

We first show that Ollivier and Forman curvature coincide on 1-cells with optimal choice of the weights of the 2-cells. Coincidence for regular unweighted graphs for which every two different cycles of length at most five share at most one edge, was recently shown in \cite{tee2021enhanced} by a cycle counting argument.

\begin{corollary}\label{cor:KequalsF}
Let $G=(X,\de,m)$ be a 1-dimensional cell complex. 
Then for $x \in C(X_1)$,
\begin{align*}
\kappa(x) = \max_K F_K(x)
\end{align*}
where the maximum is taken over all regular weighted cell complexes $K$ having $G$ as 1-skeleton.
\end{corollary}

\begin{proof}

We fix the 2-dimensional cell complex $K$ with $G$ as 1-skeleton and all cycles with length at most five as 2-cells. We set the weight of 2-cells equal one. 
Combining Theorem~\ref{thm:OllivierVectorField} and Theorem~\ref{thm:dualForman} for $\pcell=1$, we obtain
\begin{align*}
\kappa(x) = \inf_{\substack{h(x)=1\\|h|\leq 1 \\ \de x \cdot \de h\leq 0}} \de \de^*h(x) = \max_{K' \leq_\pcell K} F'(x) = \max_X F_X(x).
\end{align*}
This finishes the proof.
\end{proof}

We now show that in arbitrary dimension, we still have an estimate between Ollivier and Forman curvature.

\begin{corollary}\label{cor:FleqKappa}
Let $K=(X,\de,m)$ be a cell complex and $x \in X$.
Then,
\[
F(x) \leq \kappa(x).
\]
\end{corollary}

\begin{proof}
We estimate
\[
F(x) \leq \inf_{\substack{h(x)=1\\|h|\leq 1 \\ \de x \cdot \de h\leq 0}} \de \de^*h(x) \leq \inf_{\substack{\de f(x)=1\\|\de f |\leq 1}} \de \de^*\de f(x) = \kappa(x)
\]
where the second inequality holds as we can choose $h=\de f$ giving $\de h =0$ implying $\de x \cdot \de h \leq 0$.
This finishes the proof.
\end{proof}

\subsection{Translation between transport plans and cycle weights}
We now give an alternative proof for $\max_X F_X(x) = \kappa(x)$ for $x \in X_1$ by using Thereom~\ref{thm:NewTransportplan} and by translating transport plans to cycle weights.

\begin{proof}[Alternative proof of Corollary~\ref{cor:KequalsF}]
Let $K=(X,\de,m)$ be a cell complex containing all cycles of length at most 5 as two-cells.
Let $x=(v,w) \in X_1$. We first show how to translate transport plans to the weights of cycles containing $x$.

We recall, the set of transport plans is
\[
\Pi= (0,\infty)^{N(v,w)\times N(w,v)}
\]
with $N(v,w)=S_1(x)\setminus \{y\}$.

Let $v' \in N(v,w)$ and $w'=N(w,v)$.
If the path $\{v',v,w,w'\}$ is contained in a shortest cycle $z\in X_2$ of length at most 5, then, we write $z(v',w') = z$, and we write
$m'(z(v,w)):=\rho(v',w')m(x)$, and zero for all cycles not appearing.

The weights $m'$ together with the 1-skeleton of $K$ give a complex $K'$. We now show that the transport costs of $\rho$ coincide with the Forman curvature of $K'$.

We have
\[
F'(x) = H'x(x) - \sum_{y\neq x} |H'y(x)|
\]
If $y=(v,v')$ with $v' \in N(v,w)$, then,
$|H'y(x)| = |A_\rho(v')|$. Similarly, if $y=(w,w')$  with $w' \in N(w,v)$, then $|Hy(x)| = |B_\rho(w')|$.
Moreover,
$H'x(x) = \markov(v,w)+\markov(w,v) + \sum_{v',w'} \rho(v',w')$.
We recall $\markov(v,w)=m(v,w)/m(v)$.
Finally,
\[
\sum_{y \not\more v,w} |Hy(x)| = \sum_{v',w'}d(v',w')\rho(v',w')
\]
Summing up gives
\[
F'(x) = \markov(v,w)+\markov(w,v) + \sum_{v',w'}\rho(v',w')(1-d(v',w')) - \sum_{v'} |A(v')| - \sum_{w'} |B(w')| \leq \kappa(x).
\] 
The estimate is an equality if $\rho$ is an optimal transport plan by Theorem~\ref{thm:NewTransportplan}. Thus,
\[
\max_{K'} F_{K'}(x) \geq \kappa(x).
\]
Conversely, we can construct a transport plan from the cycle weights $m'$ via
\[
\rho(v',w') := \sum_{z \more (v',v),(v,w),(w,w')} \frac{m'(z)}{m(x)}.
\]
Similar to the calculation above, we get
\[
F_{K'}(x) \leq \markov(v,w)+\markov(w,v) + \sum_{v',w'}\rho(v',w')(1-d(v',w')) - \sum_{v'} |A(v')| - \sum_{w'} |B(w')| \leq \kappa(x).
\] 
We do not get equality in the first estimate as non-shortest cycles can have positive weight.
Choosing $K'$ optimally gives
\[
\max_{K'} F_{K'}(x) \leq \kappa(x).
\]
Together with the converse estimate, this finishes the proof.
\end{proof}

\subsection{Diameter bounds}

A fundamental consequence of uniformly positive Ricci curvature is the Bonnet Myers diameter bound. On graphs, this is well known for Ollivier curvature \cite{ollivier2009ricci} and Bakry Emery curvature \cite{liu2016bakry}, as well as their rigidity results \cite{cushing2020rigidity,liu2017rigidity}. A diameter bound in terms of entropic curvature is also known \cite[Proposition~3.4]{erbar2018poincare}, but looks somewhat different. 
Forman has given an intricate proof of a diameter bound in terms of Forman curvature in case of $m=1$ (see \cite[Theorem~6.3]{forman2003bochner}).
In this subsection, we generalize Forman's result and unify the diameter bounds.

Let $K=(X,\de,m)$ be a cell complex. For $x \in X$, we define
\[
\Deg(x) := \sum_{z\more x} \frac{m(z)}{m(x)} + \max_{v\less x}\frac{m(x)}{m(v)}.
\]
where we set the maximum to be zero if $x \in X_0$.
We write
\[D_k:=\max_{x \in X_k} \Deg(x).\]
We recall the distance bound for Ollivier curvature (see \cite[Proposition~4.14]{munch2017ollivier}) is given by
\[
d(v,w) \leq \frac{\Deg(v) + \Deg(w)}{\kappa(v,w)} \qquad \mbox{ for all } v \neq w \in X_0.
\]
We write $\diam := \max_{v,w \in X_0} d(v,w)$.
Then, the diameter bounds for both, Bakry Emery and Ollivier curvature, read as
\[
\diam \leq \frac{2D_0}R \qquad \mbox{(Ollvier, Bakry Emery)}
\]
where $R$ is either a lower bound for the Bakry Emery or the Ollivier curvature. The proofs can be done via gradient estimates for the heat equation in both cases.
In contrast Forman's diameter bound is proven via discrete Jacobi fields and reads as
\[
\diam \leq \frac{2D_1}R \qquad \mbox{(Forman curvature)}
\]
where $R$ is a lower bound for the Forman curvature in the case $m=1$.
We generalize this result to arbitrary weights by applying the Ollivier diameter bound. It is somewhat surprising that this works as $D_0$ and $D_1$ seem incompatible at first glance.
\begin{theorem}\label{thm:diameterBound}
Let $K=(X,\de,m)$ be a cell complex.
Suppose $F(x) \geq R >0$ for all $x \in X_1$. Then,
\[
\diam \leq \frac {2\min(D_1,D_0)} R. 
\]
\end{theorem}

\begin{proof}
We first observe that by Corollary~\ref{cor:KequalsF} and by the diameter bound for Ollivier curvature,
\[
\diam \leq \frac{2D_0}{\inf_{x\in X_1} \kappa(x)} \leq \frac{2D_0}{R}.
\]

We finally prove $\diam \leq \frac {2(D_1)} R$.
Let $v_0 \sim \ldots \sim v_n$ be a shortest path and assume $d(v_0,v_n) = \diam$.
We modify the weights of $1$-cells containing $v_0$ or $v_n$ via
\[
\widetilde m(x) := \begin{cases} m(v_0)\sum_{\substack{z \more x, \\z \more (v_0,v_1)}} \frac{m(z)}{m(x)} &: x \more v_0, \; x \neq (v_0,v_1), \\
m(v_n)\sum_{\substack{z \more x, \\z \more (v_{n-1},v_n)}} \frac{m(z)}{m(x)} &: x \more v_n, \; x \neq (v_{n-1},v_n), \\
m(x) &:\mbox{else}.
\end{cases}
\]
In case $\widetilde m(x)=0$, we set $\widetilde m(x)=\eps$ and let $\eps \to 0$ later.
Using the tilde will indicate that the quantity refers to $(X,\de,\widetilde m)$.
By construction, we have $\widetilde F(v_0,v_1) \geq F(v_0,v_1)$ and  $\widetilde F(v_{n-1},v_{n}) \geq F(v_{n-1},v_n)$. As $v_1$ is the only neighbor of $v_0$ on the shortest path $(v_i)_i$, we have $\widetilde F(v_i,v_{i-1}) = F(v_i,v_i-1)$ for $i=2,\ldots,n-1$.

Applying Corollary~\ref{cor:KequalsF} to $(X,\de,\widetilde m)$, we obtain $\widetilde \kappa(v_i,v_{i+1}) \geq R$ for all $i$ implying $\kappa(v_0,v_n) \geq R$ as $(v_i)_i$ is a shortest path.
The distance bound for Ollivier curvature gives
\[
\diam \leq \frac{\widetilde \Deg(v_0) + \widetilde \Deg(v_n)}R.
\]
On the other hand,
\[
\widetilde \Deg(v_0) = \sum_{x \more v_0}\frac{\widetilde m(x)}{m(v_0)} = \frac{m(v_0,v_1)}{m(v_0)} + \sum_{z \more (v_0,v_1)} \frac{m(z)}{m(v_0,v_1)} \leq \Deg(v_0,v_1) \leq D_1.
\]
Similarly, $\widetilde \Deg(v_n) \leq D_1$. Plugging in finishes the proof.
\end{proof}

\subsection{Examples}\label{sec:Example}

We consider the one-dimensional cell complex $G$ with unit weights given by $X_0 = \{1,\ldots,6\}$ and $v\sim w$ iff $|v-w| \in \{1,2\}$.
The Ollivier curvature is displayed in the figure below. We remind that we use the non-normalized Ollivier curvature meaning that the curvature can exceed two.

\includegraphics{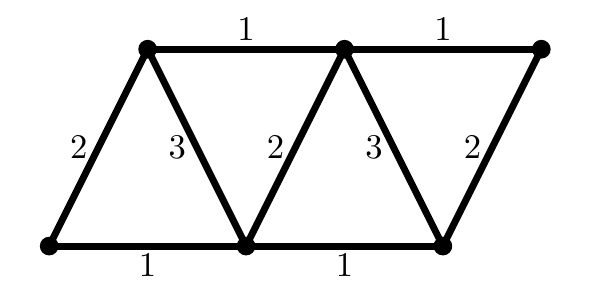}

We now compute the minimal Forman curvature maximized over the cycle weights. For a Mathematica implementation, see the appendix.
We write
$z_1$ and $z_1'$ for the outer triangles, $z_2$ and $z_2'$ for the inner triangles, $z_3$ and $z_3'$ for the outer 4-cycles, $z_4$ for the inner 4-cycle, and $z_5$ and $z_5'$ for the 5-cycles.
By symmetry and concavity of the Forman curvature in the cycle weights, we can assume $m(z_k)=m(z_k')$.
Setting 
$m(z_1)=m(z_3)=m(z_4)=\frac 2 3$ and  
$m(z_2)=\frac 1 3$ and 
$m(z_5)=0$,
and labeling the edges $(x_n)_{n=1}^9$ by
\[
x_n := \left(\left\lceil\frac n 2 \right \rceil , \left\lceil\frac {n+3} 2 \right \rceil  \right),
\]
we obtain
\[
H=\frac 1 3 
\begin{pmatrix}
10&-1&-1&-1&-2&0&0&0&0\\-1&10&1&-2&-1&-3&0&0&0\\-1&1&11&0&-2&-1&-2&0&0\\-1&-2&0&11&0&-2&-1&-3&0\\-2&-1&-2&0&12&0&-2&-1&-2\\0&-3&-1&-2&0&11&0&-2&-1\\0&0&-2&-1&-2&0&11&1&-1\\0&0&0&-3&-1&-2&1&10&-1\\0&0&0&0&-2&-1&-1&-1&10
\end{pmatrix}
\]
and thus $F \geq \frac 2 3$ with 
$
F(1,3)=F(2,4)=F(3,4)=F(3,5)=F(4,6) = \frac 2 3
$ 
and $F(1,2)=F(5,6)=\frac 5 3$ and $F(2,3) = F(4,5) = \frac 4 3$.
We now employ Theorem~\ref{thm:maxminForman} to show that $\frac 2 3$ is the optimal Forman curvature lower bound.
We set
\[
\opt := \frac 1  3 \begin{pmatrix}
0&0&0&0&0&0&0&0&0\\0&0&0&0&0&0&0&0&0\\0&0&0&0&0&0&0&0&0\\1&1&1&1&1&1&1&1&1\\1&1&1&1&1&1&1&1&1\\1&1&1&1&1&1&1&1&1\\0&0&0&0&0&0&0&0&0\\0&0&0&0&0&0&0&0&0\\0&0&0&0&0&0&0&0&0
\end{pmatrix}.
\]
Conditions $(a)$ and $(c)$ from Theorem~\ref{thm:maxminForman} are trivial to check.
For condition $(b)$, we can write $\de : C(X_1) \to C(X_2)$ as
\[
\de = \begin{pmatrix}
0&0&0&0&0&0&1&-1&1\\0&0&0&0&1&-1&1&0&0\\0&0&-1&1&-1&0&0&0&0\\1&-1&1&0&0&0&0&0&0\\0&0&0&0&1&-1&0&1&-1\\0&0&-1&1&0&-1&1&0&0\\1&-1&0&1&-1&0&0&0&0\\0&0&-1&1&0&-1&0&1&-1\\1&-1&0&1&0&-1&1&0&0
\end{pmatrix}
\]
giving $\de \opt \de^* = 0$. 
Moreover, as on $C(X_1)$, 
\[
\de \de^* = \begin{pmatrix}
2&1&-1&-1&0&0&0&0&0\\1&2&1&0&-1&-1&0&0&0\\-1&1&2&1&-1&-1&0&0&0\\-1&0&1&2&1&0&-1&-1&0\\0&-1&-1&1&2&1&-1&-1&0\\0&-1&-1&0&1&2&1&0&-1\\0&0&0&-1&-1&1&2&1&-1\\0&0&0&-1&-1&0&1&2&1\\0&0&0&0&0&-1&-1&1&2
\end{pmatrix},
\]
we get $\Tr(\de \de^* \opt) = \frac 2 3$.
This shows
\[
\max_K \min_{x \in X_1} F_K(x) = \frac 2 3
\]
where the maximum is taken over all cell complexes $K$ having $G$ as $1$-skeleton.
This demonstrates that the optimal lower Forman curvature bound can be strictly smaller than the minimal Ollivier curvature, and non-integer valued even in the case $m=1$.

\section{Varying the distance  -- Summarizing the results}\label{sec:VaryDistance}
For defining Ollivier curvature, one has the freedom of choosing a distance on $X_0$. For being compatible with Forman curvature, we restrict ourselves to path distances compatible with the edges.
Let $K=(X,\de,m)$ be a cell complex. We recall $\de:C(X_\pcell) \to C(X_{\pcell +1})$ is the (unweighted) coboundary operator, and $m$ induces a scalar product on $C(X)$.
Let $\omega: X \to (0,\infty)$.
We define the path distance $d_\omega:X_0 \to [0,\infty]$ via
\[
d_\omega(v,w) := \inf \left\{\sum_{k=1}^n \omega(v_k,v_{k-1}) : v=v_0\sim \ldots,\sim v_n = w \right\}.
\]
We will always assume that $\omega$ is \emph{non-degenerate}, i.e., every edge $x=(v,w) \in X_1$ is the unique shortest path between its vertices.
The Ollivier curvature $\kappa_\omega$ of $x \in X$ with respect to $\omega$ is given by
\[
\omega(x)\kappa_\omega(x) := \inf_{\substack{\de f(x) = \omega(x)\\|\de f| \leq \omega}} \de \de^*\de f.
\]
If $x \in X_1$, then $\kappa_\omega(x)$ is the usual Ollivier curvature of the edge $x$ with respect to $d_\omega$, compare \cite[Theorem~2.1]{munch2017ollivier}.
For $v,w \in X_0$, we have $\kappa_\omega(v,w) = \infty$ iff $d_\omega(v,w) <\omega(v,w)$ for $v,w \in X_0$,
as the infimum is taken over the empty set.
Therefore, non-degenerated $\omega$ means finite Ollivier curvature on $X_1$.
To make Forman curvature compatible with Ollivier curvature with repect to $\omega$, we define for $x \in X$,
\[
F_\omega (x) :=  Hx(x) - \sum_{y \neq x} \frac{\omega(y)}{\omega(x)}|Hy(x)|.
\]
We recall $H=\de^*\de + \de \de^*$. For $x \in X_1$, the weighted Forman curvature $F_\omega$ reads as
\[
F_\omega (x) = \sum_{v \less x}\frac{m(x)}{m(v)} + \sum_{z \more x} \frac{m(z)}{m(x)} - \sum_{y\neq x} \frac {\omega(y)}{\omega(x)}\left|   \sum_{v \less x,y} \frac{m(y)}{m(v)} - \sum_{z\more x,y}  \frac{m(z)}{m(x)}  \right|.
\]

We notice that the constraint $|\de f|\leq \omega/\omega(x)$ in the Ollivier curvature becomes a coefficient $\omega(y)/\omega(x)$ in the Forman curvature, as one would expect from dual linear programs.

In terms of diagonally dominant operators and the diagonal split off $D(H)$ from \eqref{eq:diagPlusSignedLaplace}, we can write \[\omega F_\omega = D(\omega H)\] 
where $F_\omega$ is interpreted as multiplication operator.
In the following, we  state our results about Forman and Ollivier curvature with arbitrary non-degenerated $\omega$.
As they are straight forward generalizations of the case $\omega=1$, we will omit the proofs and refer to the corresponding results with $\omega=1$.


\subsection{Semigroup estimates}
We give semigroup estimates for $\kappa_\omega$ and $F_\omega$.
To do so, we recall the supremum norm relative to $\omega$, given by
\[
\|f\|_{\omega,\infty} = \left\|\frac f \omega \right\|_\infty,
\]
and the $\pnorm$-norm for $\pnorm \in [1,\infty)$, given by
\[
\|f\|_{\omega,\pnorm} = \left\| \frac{f}{\omega^{1-2/\pnorm}} \right\|_\pnorm.
\]

\begin{theorem}[see Theorem~\ref{thm:HodgeSemigroup} and Theorem~\ref{thm:weightedDiagonalDominance}]\label{thm:HodgeSemigroupOmega}
Let $(X,\de,m)$ be a cell complex and let $\omega:X \to (0,\infty)$ be non-degenerate. let $\pcell \in \N$ and $\Ri \in \R$.
The following statements are equivalent:

\begin{enumerate}[(i)]
\item $F_\omega(x) \geq \Ri$ for all $x \in X_\pcell$,
\item $\left\|e^{-tH} f \right\|_{\omega,\infty} \leq e^{-\Ri t} \left\| f  \right\|_{\omega,\infty}$ for all $f \in C(X_\pcell)$.
\item $\left\|e^{-tH} f \right\|_{\omega,\pnorm} \leq e^{-\Ri t} \left\| f  \right\|_{\omega,\pnorm}$ for all $f \in C(X_\pcell)$ and all $p \in [1,\infty]$,
\item $H_\pcell-R$ is diagonally dominant with respect to $\omega$.
\end{enumerate}
\end{theorem}

The next theorem is a generalization of the Lipschitz contraction characterization of Ollivier curvature, see \cite{munch2017ollivier,veysseire2012coarse}. In general, $\|\de f\|_{\infty,\omega}$ cannot be expressed as Lipschitz constant of $f$, making the following theorem conceptually interesting.

\begin{theorem}\label{thm:OllivierSemigroupOmega}
Let $(X,\de,m)$ be a cell complex and $\omega:X \to (0,\infty)$ non-degenerate, let $\pcell \in \N$ and $\Ri \in \R$.
The following statements are equivalent:

\begin{enumerate}[(i)]
\item $\kappa_\omega(x) \geq \Ri$ for all $x \in X_\pcell$,
\item $\left\| \de e^{-t\de^*\de} f \right\|_{\omega,\infty} \leq e^{-\Ri t} \left\|\de f \right\|_{\omega,\infty}$ for all $f \in C(X_{\pcell-1})$.
\end{enumerate}
\end{theorem}
\begin{proof}
We notice that with $h:= e^{-t\de^*\de} f$,
\[
\partial_t^+ \log  \left\| \de e^{-t\de^*\de} f \right\|_{\omega,\infty} =  -\min_{\de h(x)=\|\de h\|_{\infty,\omega}}  \frac{\de \de^*\de h(x)}{\|\de h\|_{\omega,\infty}}  \leq -\min_{x} \kappa_\omega(x)
\]
implying $(i) \Rightarrow (ii)$ by integrating.
On the other hand, equality can be achieved for $t=0$ implying $(ii) \Rightarrow (i)$. This finishes the proof.
\end{proof}

The theorem can be seen as a translation between maximum principle and contraction properties of a semigroup, similar to the semigroup estimate for the Forman curvature.

\subsection{Forman and Ollivier curvature as dual linear programs}

In order to translate between Forman and Ollivier curvature, we give the dual linear program to maximizing the Forman curvature on $X_\pcell$ over cell weights on $X_{\pcell + 1}$. 
We recall $K' \leq_\pcell K$ if
$X_\kcell = X'_\kcell$ for $\kcell\leq \pcell$, and 
$X_\kcell' \subseteq X_\kcell$ for $\kcell >\pcell$, and
$m(x)=m'(x)$ whenever $\dim(x) \leq \pcell$, and
 $\de'v(x)=\de v(x)$ for all $v,x \in X'$.

\begin{theorem}[see Theorem~\ref{thm:dualForman}]\label{thm:dualFormanOmega}
Let $K=(X,\de,m)$ be a cell complex and $\omega:X \to (0,\infty)$ non-degenerate, and $x \in X_\pcell$ for some $\pcell \in \N_0$.
Then,
\[
\max_{K' \leq_\pcell K}\omega(x)F'_\omega(x) = \min_{\substack{h(x)=\omega(x)\\|h|\leq \omega\\\de h \cdot \de x \leq 0}} \de \de^*h(x).
\]

\end{theorem}

We now give the dual linear program to maximizing the minimal Forman curvature of $\pcell$-cells when maximizing over the cell weights on $X_{\pcell + 1}$. We recall $\Tr (\opt) = \sum_{x \in X_\pcell} \opt x(x)$ for linear $\opt :C(X_\pcell) \to C(X_\pcell)$.

\begin{theorem}[see Theorem~\ref{thm:maxminForman}] \label{thm:maxminFormanOmega}
Let $K=(X,\delta,m)$ be a cell complex and $k \in \N_0$. 
Let $\omega:X \to (0,\infty)$ be non-degenerate.
Then,
\[
\max_{K' \leq_k K} \min_{x \in X_k}F_\omega'(x) = \min_\opt \Tr(\de \de^* \opt)
\]
where the minimum is taken over all linear $\opt \in C(X_k) \to C(X_k)$ satisfying
\begin{enumerate}[(a)]
\item $\opt f(x) \geq 0$ whenever $2\langle x, \omega f\rangle \geq \|f\|_{\omega,1}$,
\item $\de \opt \de^* z(z) \leq 0$ for all $ z \in X_{k+1}$,
\item $\Tr(\opt) =1$.
\end{enumerate}

\end{theorem}

We now show that the Ollivier curvature on $X_1$ can be expressed similarly to Theorem~\ref{thm:dualFormanOmega}.
We say a cycle $z\more x \in X_1$ is shortcutting if 
\[
\sum_{y\less z} \omega(y)< 2\sum_{\substack{y\less z\\ \de \de^*x(y) \neq 0}} \omega(y). \]
We remark that in the case $\omega=1$, shortcutting means length at most five. If $\omega \neq 1$, then shortcutting of a cycle $z$ also depends on the edge $x \less z$.
The relevance of shortcutting cycles is to decide whether $d_\omega(v',w')=\omega(v',v)+\omega(v,w) + \omega(w,w')$. Namely this equality only holds iff there is no shortcutting cycle $z \more (v,w)$.

\begin{theorem}[see Theorem~\ref{thm:OllivierVectorField}]\label{thm:OllivierVectorFieldOmega} Let $(X,\de,m)$ be a cell complex and $\omega:X \to (0,\infty)$ non-degenerate. Let $x \in X_1$.
Suppose all shortcutting cycles containing $x$  are 2-cells. 
Then,
\[
\omega(x)\kappa_\omega(x) = 
\inf_{\substack{h(x)=\omega(x)\\|h|\leq \omega \\ \de x \cdot \de h\leq 0}} \de \de^*h(x).
\]
For the infimum, we only need to consider $h \in C(X_1)$. 
\end{theorem}

\begin{corollary}[see Corollary~\ref{cor:KequalsF}]\label{cor:KequalsFOmega}
Let $G=(X,\de,m)$ be a 1-dimensional cell complex and $\omega :X \to (0,\infty)$ non-degenerate. 
Then for $x \in C(X_1)$,
\begin{align*}
\kappa_\omega(x) = \max_{K} F_{K,\omega}(x)
\end{align*}
where the maximum is taken over all cell complexes $K$ having $G$ as 1-skeleton.
\end{corollary}
Alternatively, one can also take the maximum over all complexes having only cycles shortcutting $x$ as 2-cells, and $G$ as 1-skeleton.

Having shown identity of Forman and Ollivier curvature on $X_1$, we now show $F \leq \kappa$ for arbitrary cells.

\begin{corollary}[see Corollary~\ref{cor:FleqKappa}]
Let $K=(X,\de,m)$ be a cell complex and $\omega:X \to (0,\infty)$ non-degenerate and $x \in X$.
Then,
\[
F_\omega(x) \leq \kappa_\omega(x).
\]
\end{corollary}

\subsection{Diameter bounds}

Let $K=(X,\de,m)$ be a cell complex.
We recall that for $x \in X$, we have
\[
\Deg(x) = \sum_{z\more x} \frac{m(z)}{m(x)} + \max_{v\less x}\frac{m(x)}{m(v)} \qquad \mbox{ and } \qquad  D_k =\max_{x \in X_k} \Deg(x)
\]
where we set the maximum to be zero if $x \in X_0$.
We write $\diam_\omega := \max_{v,w \in X_0} d_\omega(v,w)$.
We now give the diameter bound.
\begin{theorem}[see Theorem~\ref{thm:diameterBound}]\label{thm:diameterOmega}
Let $K=(X,\de,m)$ be a cell complex and $\omega:X \to (0,\infty)$ non-degenerate.
Suppose $F_\omega(x) \geq R >0$ for all $x \in X_1$. Then,
\[
\diam_\omega \leq \frac {2\min(D_1,D_0)} R. 
\]
\end{theorem}

\subsection{Original Forman curvature as a special case}\label{sec:OriginalForman}

Forman defines the curvature for a cell complex with cell weights $w$ as
\[
\frac{F_{\mbox{original}}(x)}{w(x)} := \sum_{v \less x}\frac{w(v)}{w(x)} + \sum_{z \more x} \frac{w(x)}{w(z)} - \sum_{y\neq x} \left|   \sum_{v \less x,y}  \frac{w(v)}{\sqrt{w(y)w(x)}} - \sum_{z\more x,y}  \frac{\sqrt{w(x)w(y)}}{w(z)}  \right|
\]
for $x \in X_1$.
Our curvature with cell weights $m$ is given by
\begin{align*}
F_\omega (x) = \sum_{v \less x}\frac{m(x)}{m(v)} + \sum_{z \more x} \frac{m(z)}{m(x)} - \sum_{y\neq x} \frac {\omega(y)}{\omega(x)}\left|   \sum_{v \less x,y} \frac{m(y)}{m(v)} - \sum_{z\more x,y}  \frac{m(z)}{m(x)}  \right|.
\end{align*}
With the choices
\begin{itemize}
\item $m:=1/w$,
\item $\omega := {\sqrt{w}}$,
\end{itemize}
we have

\[
F_\omega (x) = \sum_{v \less x}\frac{w(v)}{w(x)} + \sum_{z \more x} \frac{w(x)}{w(z)} - \sum_{y\neq x} \frac {\sqrt{w(y)}}{\sqrt{w(x)}}\left|   \sum_{v \less x,y} \frac{w(v)}{w(y)} - \sum_{z\more x,y}  \frac{w(x)}{w(z)}  \right| = \frac{F_{\mbox{original}}(x)}{w(x)}.
\]
This identity generalizes to higher order cells by taking the signs of $\de$ into account.
The choice $m=1/w$ comes from the fact that the coboundary operator definitions differ by a factor of the cell weights.
The choice $\omega=\sqrt{w}$ comes from the fact that Forman uses a specific change of basis making the Hodge Laplacian selfadjoint with respect to the constant weight.

\printbibliography

\appendix
\section{Mathematica source code for computing Forman curvature}
We include the source code of a Mathematica \cite{WMathematicaOnline} program to calculate the optimal minimal Forman curvature.
Construction of the graph given in the example (see Section~\ref{sec:Example}):
\begin{verbatim}
Adj[x_,y_]  =  KroneckerDelta[Abs[x-y],1] + KroneckerDelta[Abs[x-y],2];
A           =  Table[Adj[i,j],{i,1,6},{j,1,6}];
G           =  AdjacencyGraph[A];
\end{verbatim}
The next step is to label the cells and define the coboundary operator $\delta$ called d0 and d1 respectively. This part has to be run before one can compute the curvature. Outputs are the graph with labeled vertices and edges, as well as the list of cycles of length at most five:
\begin{verbatim}
GraphPlot[G, VertexLabels -> Table[i -> Subscript[v, i], {i, VertexCount[G]}],
             EdgeLabels->Table[EdgeList[G][[i]]->Subscript[x,i],{i,EdgeCount[G]}]]
GDirected                 =   DirectedGraph[G,"Acyclic"];
IncList                   =   IncidenceList[GDirected,VertexList[G]];
CycleList                 =   FindCycle[G,5,All]
EdgeToIncidence[x_,y_]   :=  (KroneckerDelta[x[[1]],y[[1]]]*KroneckerDelta[x[[2]],y[[2]]] 
                            - KroneckerDelta[x[[1]],y[[2]]]*KroneckerDelta[x[[2]],y[[1]]])
EdgeToIncidenceSum[x_,y_]:=   Sum[EdgeToIncidence[x[[k]],y],{k,Length[x]}]
d0                        =   Transpose[IncidenceMatrix[GDirected]];
d1                        =   Table[EdgeToIncidenceSum[CycleList[[i]],IncList[[j]]],
                              {i,Length[CycleList]},{j,Length[IncList]}];
\end{verbatim}
Direct computation of the optimal lower Forman curvature bound. The outputs are the curvature bound and the list of cycle weights:
\begin{verbatim}
Hodge[m2var_]   :=  Transpose[d1].(m2var*d1)  + d0.Transpose[d0]
Forman[m2var_]  :=  2*Diagonal[Hodge[m2var]] - Total[Abs[Hodge[m2var]]]
m2List           =  Array[m2,Length[CycleList]];
MaxFormanH       =  NMaximize[{Min[Forman[m2List]],m2List>=0},m2List];
MaxFormanH[[1]]
Chop[m2List/.MaxFormanH[[2]],10^-8]
\end{verbatim}
Computation of the optimal lower Forman curvature bound via Theorem~\ref{thm:maxminForman}. Outputs are curvature bound and the operator $J$:
\begin{verbatim}
J           =  Array[j,{EdgeCount[G],EdgeCount[G]}];
a1          =  Diagonal[J] -J;
a2          =  Diagonal[J] + J;
b           =  Diagonal[d1.J.Transpose[d1]];
c           =  Tr[J];
MaxFormanJ  =  NMinimize[{Tr[d0.Transpose[d0].J],a1>=0,a2>=0,b<=0,c==1},Flatten[J]];
MaxFormanJ[[1]]
MatrixForm[Chop[J/.MaxFormanJ[[2]]]]
\end{verbatim}

\end{document}